\documentclass[12pt]{amsart}
\usepackage[usenames]{color} 
\usepackage{times}
\usepackage{amsfonts}
\usepackage{amsthm}
\usepackage{amsmath}
\usepackage{psfrag}

\usepackage{times}
\usepackage{latexsym,color}
\usepackage{amssymb}
\usepackage{graphicx}

\usepackage{epsfig}

\usepackage{gastex}

\usepackage{tikz}
\usepackage[all]{xy}
\usetikzlibrary{decorations}
 \usetikzlibrary{decorations.pathreplacing}
 \usepackage[hang]{caption}
 \usetikzlibrary{snakes}
\usetikzlibrary{arrows}

\advance \topmargin by -\headheight
\advance \topmargin by -\headsep
\evensidemargin \oddsidemargin
\def \ind{_{n \in {\mbox{\rm {\scriptsize I$\!$N}}}}}

\newcommand{\GR}{{\mathbb R}}
\newcommand{\GZ}{{\mathbb Z}}
\newcommand{\GC}{{\mathbb C}}
\newcommand{\GN}{{\mathbb N}}

\newcommand{\ab}{|}

\newtheorem{theorem}{Theorem}
\newtheorem{lemma}[theorem]{Lemma}

\newtheorem{proposition}[theorem]{Proposition}
\newtheorem{question}{Question} 

\newtheorem{definition}{Definition}
\newtheorem{remark}{Remark}

\begin{document}

\title[Rigidity]{Rigidity of square-tiled interval exchange transformations}

\author[S. Ferenczi]{S\'ebastien Ferenczi} 
\address{Aix Marseille Universit\'e, CNRS, Centrale Marseille, Institut de Math\' ematiques de Marseille, I2M - UMR 7373\\13453 Marseille, France.}
\email{ssferenczi@gmail.com}
\author[P. Hubert]{Pascal Hubert} 
\address{Aix Marseille Universit\'e, CNRS, Centrale Marseille, Institut de Math\' ematiques de Marseille, I2M - UMR 7373\\13453 Marseille, France.}
\email{hubert.pascal@gmail.com}

\date{February 17, 2017}

\begin{abstract}
We look at interval exchange transformations defined as first return maps on the set of diagonals of a flow of direction $\theta$ on a square-tiled surface: using a combinatorial approach, we show that, when the surface has at least one true singularity  both the flow and the interval exchange  are  rigid if and only if  $\tan\theta$ has bounded partial quotients. Moreover, if all vertices of the squares are singularities of the flat metric, and $\tan\theta$ has bounded partial quotients,  the square-tiled
 interval exchange transformation $T$ is not of rank one. Finally, for another class of surfaces, those defined by the unfolding of billiards in Veech triangles,  we build an uncountable set of rigid directional flows and an uncountable set of rigid interval exchange transformations.

\end{abstract}
\maketitle

\begin{center}
{\it To the memory of William Veech whose mathematics were a constant source of inspiration for both authors, and who always showed great kindness to the members of the Marseille school, beginning with its founder G\'erard Rauzy.}
\end{center}

\vskip 1cm

Interval exchange transformations were originally
introduced by Oseledec
\cite{ose}, following an idea of Arnold \cite{arn}, see also Katok and Stepin \cite{ks}; an
exchange
of $k$ intervals, denoted throughout this paper by $\mathcal I$, is given by a positive vector of $k$ lengths together with a
permutation $\pi$ on $k$ letters; the unit interval is partitioned
into $k$ subintervals of lengths $\alpha _1,\ldots ,\alpha_k$ which are
rearranged by $\mathcal I$ according to $\pi$.

 The history of interval exchange transformations is made with big {\em generic} results: almost every interval 
 exchange transformation is uniquely ergodic (Veech \cite{vue}, Masur \cite{mas}), almost every interval exchange transformation
 is weakly mixing (Avila-Forni \cite{af}), while other results like simplicity \cite{vs} or Sarnak's conjecture \cite{sar} are still
 in the future.  In parallel with generic results, people have worked to build
 constructive examples, and, more interesting and more difficult, counter-examples. In the present paper
 we want to focus on two less-known but very important generic results, both by Veech: almost every interval exchange transformation 
 is {\em rigid} \cite{vs}, almost every interval exchange transformation
 is of  {\em rank one} \cite{vr1}. 
 
 These results are not true for every interval exchange transformation.
 The last result admits already a wide collection of examples and counter-examples, as indeed 
 the 
 first two papers ever written on interval exchange transformations provide counter-examples to a weaker property (Oseledec \cite{ose}) and
 examples of 
 a stronger property (Katok-Stepin \cite{ks}); in more recent times, many examples were built, such 
 as most of those in \cite{fz2} \cite{fie}, and also a surprisingly vast amount of counter-examples, as, 
 following Oseledec, 
 many great minds built interval exchange transformations with given spectral multiplicity functions, for example Robinson \cite{robi} or Ageev \cite{ag} and
 this contradicts rank one as soon as the latter 
 is not constantly one (simple spectrum); let us just remark that these brilliant examples, built on purpose,
 are a little complicated and not very explicit as interval exchange transformations. We know of only one family of interval exchange transformations which
 have simple spectrum but not rank one, these were built in \cite{bcf} but only for $3$ intervals.
 
 As for the question of rigidity, it has been solved completely  for the case
 of $3$-interval exchange transformations in \cite{fhz4}, where a necessary and sufficient condition 
 is given. For more than three intervals, examples of rigidity can again be found in \cite{fz2} \cite{fie}.\\
  
  But of course, possibly the main appeal of interval exchange transformations is the fact that they are closely linked to linear flows on translation surfaces, which are studied using Teichm\"uller dynamics. 
 Generic results are obtained applying the  $SL(2,\GR)$ action on translation surfaces.
 After all the efforts made to classify $SL(2,\GR)$ orbit closures in the moduli spaces of abelian differentials, especially after the work of Eskin, Mirzakhani and Mohammadi \cite{esmi0, esmi}, it is quite natural to want to solve these ergodic questions on suborbifolds of moduli spaces. The celebrated Kerckhoff-Masur-Smillie Theorem \cite{kms} solved the unique ergodicity question for every translation surface and almost every direction. Except for this general result, very little is known on the ergodic properties of linear flows and interval exchange transformations obtained from suborbifolds. Avila and Delecroix recently proved that, on a non arithmetic Veech surface, in  a generic direction, the linear flow is weakly mixing \cite{ad}.
 
 In the present paper, we shall study two families of Veech surfaces, the {\em square-tiled} surfaces, and the  surfaces built by unfolding {\em billiards in Veech triangles}.
 
  In Teichm\"uller dynamics, square-tiled surface play a special role since they are integers points in period coordinates. Moreover, the $SL(2,\GR)$ orbit of a square-tiled surface is closed in its moduli space. The main part of the present paper studies families of interval exchange transformations associated with square-tiled surfaces.
 Our main results are:
 \begin{theorem} \label{thm:flow}
 Let $X$ be a square-tiled surface of genus at least 2. 
 The linear flow in direction $\theta$ on $X$ is rigid if and only if the slope $\tan \theta$ has unbounded partial quotients. 
 \end{theorem}
 
 \begin{remark} The new and more difficult statement in Theorem \ref{thm:flow} is the non rigidity phenomenon when the slope has bounded partial quotients. 
 \end{remark}
 
 Theorem \ref{thm:flow} can be restated in terms of interval exchange transformations. Given a square-tiled surface and a direction with positive slope $\tan \theta$,  defining
 $\alpha = \frac{1}{ 1 + \tan \theta}$,
 there is very natural way to associate an interval exchange transformation $T_{\alpha}$, namely the first return map on the union of the diagonals of slope $-1$  of the squares (the length of diagonals is normalized to be 1). It is a finite extension of a rotation of angle $\alpha$, and an interval exchange transformation on a multi-interval. We call it a {\em square-tiled interval exchange transformation}.  
 \begin{theorem} \label{thm:main}
 Let $X$ be a square-tiled surface of genus at least 2. 
 The square-tiled interval exchange transformation $T_\alpha$ is rigid if and only if $\alpha$ has unbounded partial quotients\footnote{$\alpha$ has bounded partial quotients if and only if  $\tan \theta$ has bounded partial quotients}. 
 \end{theorem}
 
 \begin{remark}
 To our knowledge, these examples are the first appearance of non rigid interval exchange transformations on more than 3 intervals, together with the examples defined simultaneously by Robertson \cite{robe}, where a different class of interval exchanges is shown to have the stronger property of {\em mild mixing} (no rigid factor). Our examples are not weakly mixing, and therefore not mildly mixing. Note also that Franczek \cite{fra}
 proved that mildly mixing  flows are dense in genus at least two, and that Kanigowski and Lema\'nczyk \cite{kale} proved that mild mixing is implied by {\em Ratner's property}, which thus our examples do not possess.
 \end{remark}
 
 \begin{question} Do there exist interval exchanges which are weakly mixing, not rigid but not mildly mixing? \end{question}

\begin{question} Find interval exchanges satisfying Ratner's property (note that the examples of Robertson are likely candidates). \end{question}

\begin{question}(Forni) Is a  {\em self-induced} interval exchange always non-rigid  when the permutation is not circular? \end{question}

 The above theorem, again in the direction of non rigidity, constitutes the main result of the paper; its proof relies on the word combinatorics of the natural coding of the interval exchange. Indeed, rigidity of a symbolic system translates, through the ergodic theorem, into  a form of approximate periodicity on the words: the iterates by some sequence $q_n$ of a very long word $x=x_1\ldots x_k$ should be words arbitrarily close to $x$ in the Hamming distance $\bar d$; to deny this property, the known methods consist either in showing that there are many possible $T^{q_n}x$ (thus for example {\em strong mixing} contradicts rigidity), but this will not be the case here, or else, as was initiated by Lema\' nczyk and Mentzen  \cite{leme}, in showing that $\bar d$-neighbours are scarce, and thus our appronximate periodicity forces periodicity, which is then easy to disprove.
But the examples of \cite{leme}, including some well-known systems like the Thue - Morse subshift (del Junco \cite{de}), satisfy a strong property on the scarcity of $\bar d$-neighbours, namely Proposition \ref{lmr} below with $e=1$ (two close enough neighbours must actually coincide on a connected central part); this property, which is shared also by Chacon's map, is {\em not} satisfied in general by our interval exchanges, see Remark \ref{ctex} below; they do satisfy a weaker property, namely Proposition \ref{lmr} under its general form, involving averages on a finite number of orbits,  which seems completely new and is sufficient to complete the proof of non-rigidity.

The stronger property on the scarcity of $\bar d$-neighbours is satisfied in some particular cases, and we use it to prove

 \begin{theorem}\label{thm:rank}
 If all vertices of the squares are singularities of the flat metric, and $\alpha$ has bounded partial quotients,  the square-tiled
 interval exchange transformation $T_\alpha$ is not rank one.
 \end{theorem}
 
 \begin{remark}
 This condition is very restrictive and only holds for a finite number of square-tiled surfaces in each stratum.
 \end{remark}

In the last part,  we exhibit an uncountable set of rigid directional flows (see Proposition \ref{prop-flow-billiard}) and an uncountable  set of rigid interval exchange transformations (see  Proposition  \ref{prop:interval exchange-billiard}) associated with the unfolding of billiards in Veech triangles; in these examples, the directions are well approximated by periodic ones.

 \begin{remark} 
The proof of Proposition \ref{prop-flow-billiard} works mutatis mutandis for every Veech surface. 
\end{remark}

\begin{question}
On a primitive Veech surface, is the translation flow in a typical direction rigid?
\end{question}

 \subsection{Organization of the paper}
 In Section \ref{sec:def} we recall the classical definitions about interval exchange transformations, coding, square tiled surfaces and some facts in ergodic theory. Section \ref{sec:square-tiled-interval exchange} presents square tiled interval exchange transformations and their symbolic coding. 
In Section \ref{sec:proof}, we give a proof of Theorem \ref{thm:main} using combinatorial methods; the main tool is Proposition \ref{lmr}. In Section \ref{sec:consequences}, we deduce from Theorem \ref{thm:main} a proof of Theorem \ref{thm:flow}. We also prove Theorem \ref{thm:rank}.
In Section \ref{sec:reg-billiards}, we tackle the case of billiards in Veech triangles.

\section{Definitions} \label{sec:def}
\subsection{Interval exchange transformations}
For any question about interval exchange transformations, we refer the reader to the surveys \cite{via} \cite{yoc}. Our intervals are always semi-open, as $[a,b[$.

\begin{definition}\label{interval exchange} A {\em $k$-interval exchange transformation}
 $T$
with vector $(\alpha _1,\alpha _2,\ldots ,\alpha _k)$,
 and
permutation $\pi$ is
defined on $[0,\alpha_1+\ldots \alpha_k[$ by
$$
{\mathcal I}x=x+\sum_{\pi^{-1}(j)<\pi^{-1}(i)}\alpha_{j}-\sum_{j<i}\alpha_{j}.
$$
when $x$ is in the interval
$$\left[ \sum_{j<i}\alpha_{j}
,\sum_{j\leq i}\alpha_{j}\right[.$$

\noindent We put  
$\gamma_{i}=\sum_{j\leq i}\alpha_{j}$, and denote by  $\Delta_i$  
 the interval $[\gamma_{i-1}, \gamma_{i}[$ if $2\leq i\leq k-1$, 
 while $\Delta_1=[0, \gamma_{1}[$ and $\Delta_k=[\gamma_{k-1}, 1[$.
 \end{definition}

\subsection{Word combinatorics}
We look at finite {\em words} on a finite alphabet ${\mathcal A}=\{1,...k\}$. A word $w_1...w_t$ has
{\em length} $\ab w\ab=t$ (not to be confused with the length of a corresponding interval). The {\em  empty word} is  the unique word of length $0$.  The {\em concatenation} of two words $w$ and $w'$ is denoted by $ww'$. 

\begin{definition}\label{dln} 
\noindent A word $w=w_1...w_t$ {\em occurs at place $i$} in a word $v=v_1...v_s$ or an infinite sequence  $v=v_1v_2...$ if $w_1=v_i$, ...$w_t=v_{i+t-1}$. We say that $w$ is a {\em factor} of $v$. The empty word is a factor of any $v$. {\em Prefixes} and {\em suffixes} are defined in the usual way.
\end{definition}

\begin{definition}
A {\em   language} $L$ over $\mathcal A$ is a set of words such if $w$ is in $L$, all its factors are in $L$,  $aw$ is in $L$ for at least one letter $a$ of $\mathcal A$, and $wb$ is in $L$ for at least one letter $b$ of $\mathcal A$.

\noindent A language $L$ is
{\em  minimal} if for each $w$ in $L$ there exists $n$ such that $w$
occurs in each word  of $L$ with $n$ letters.

\noindent The language $L(u)$ of an infinite sequence $u$ is the set of its finite factors.
\end{definition}

\begin{definition} For two words of equal length $w=w_1\ldots w_q$ and $w'=w'_1\ldots w'_q$, their ${\bar d}$-distance is 
${\bar d}(w,w')=\frac{1}{q}\#\{i; w_i\neq w'_i\}$.
\end{definition}

\begin{definition}\label{ad}
 A word $w$ is called {\em  right special}, resp. {\em
left
special} if there are at least two different letters $x$ such that $wx$, resp. $xw$, is in $L$. If $w$ is both right special and
left special, then $w$ is called {\em  bispecial}. \end{definition}

\subsection{Codings}

\begin{definition} The {\em symbolic dynamical system} associated to a language $L$ is the one-sided shift $S(x_0x_1x_2...)=x_1x_2...$ on the subset $X_L$ of ${\mathcal A}^{\GN}$ made with the infinite sequences such that for every $t<s$, $x_t...x_s$ is in $L$.

\noindent For a word $w=w_1...w_s$ in $L$, the {\em cylinder} $[w]$ is the set $\{x\in X_L; x_0=w_1, ... x_{s-1}=w_s\}$. 
\end{definition}

Note that the symbolic dynamical system $(X_L,S)$ is minimal (in the usual sense, every orbit is dense) if and only if the language $L$ is mimimal.

\begin{definition}\label{sy} For a system $(X,T)$ and a finite partition $Z=\{Z_1,\ldots Z_r\}$ of $X$, the {\em trajectory} of a point $x$ in $X$ is the infinite
sequence
$(x_{n})\ind$ defined by $x_{n}=i$ if ${T}^nx$ falls into
$Z_i$, $1\leq i\leq r$.

\noindent Then $L(Z,T)$ is the language made of all the finite factors of all the  trajectories, and $X_{L(Z,T)}$ is the {\em coding} of $X$ by $Z$. 

\noindent For an interval exchange transformation $T$, if we take for $Z$ the partition made by the intervals $\Delta_i$, $1\leq i\leq k$, we denote $L(Z,T)$ by $L(T)$ and call $X_{L(T)}$ {the \em natural coding} of $T$.
\end{definition}

\subsection{Measure-theoretic properties}
\begin{definition} $(X, T, \mu )$ is {\em rigid} if there exists a sequence
    $q_{n}\to\infty$  such that for
any measurable set $A$ 
 $\mu (T^{q_{n}}A\Delta A)\to 0.$ \end{definition} 

\begin{definition}\label{nm} In $(X,T)$, a Rokhlin {\em tower}  is a collection 
    of disjoint measurable sets called {\em levels} $F$, $TF$, \ldots, $T^{h-1}F$.  If $X$ is 
    equipped with a partition $P$ such that each level $T^{r}F$ is contained in
    one atom
$P_{w(r)}$, 
 the {\em name} of the 
tower is the word
 $w(0)\ldots w(h-1)$.\end{definition}

  \begin{definition}\label{dr1} A system $(X, T,\mu )$  is of {\em rank one} if there exists a sequence of Rokhlin towers such that the whole $\sigma$-algebra is generated by the partitions $\{F_n, TF_n, \ldots, T^{h_n-1}F_n, X\setminus.\cup_{j=0}^{h_n-1}T^jF_n\}$.
      \end{definition}

\subsection{Translation surfaces and square-tiled surfaces} \label{sec:square-tiled-surfaces}

A translation surface is an equivalence class of polygons in the plane with edge identifications: Each translation surface is a finite union of polygons in $\GC$, together with a choice of pairing of parallel sides of equal length.  Two such collections of polygons are considered to define the same translation surface if one can be cut into pieces along straight lines and these pieces can be translated and re-glued to form the other collection of polygons (see Zorich \cite{zo}, Wright \cite{wr} for surveys on translation surfaces). For every direction $\theta$, the linear flow in direction $\theta$ is well defined. The first return map to a transverse segment is an interval exchange. \\

Recall that closed regular geodesics on a flat surface appear in families of parallel closed geodesics. Such families cover a cylinder filled with parallel closed geodesic of equal length. Each boundary of such a cylinder contains a singularity of the flat metric. \\

A square-tiled surface is a finite collection of unit squares $\{1, \dots, d\}$, the left side of each square is glued by translation to the right side of another square. The top of a square is glued to the bottom of another square. A baby example is the flat torus $\GR^2/\GZ^2$. In fact every square-tiled surface is a covering of $\GR^2/\GZ^2$ ramified at most over the origin of the torus. 
A square-tiled surface is a translation surface, thus linear flows are well defined.
 Combinatorially, a square-tiled surface is defined by two permutations acting on the squares: $\tau$ encodes horizontal identifications, $\sigma$ is responsible for the vertical identifications. For $1 \leq i \leq d$,
 $\tau(i)$ is the square on the right of $i$ and $\sigma(i)$ is the square on top of $i$.
 The singularity of the flat metric are the projections of some vertices of the squares with angles $2k\pi$ with $k>1$. The number $k$ is explicit in terms of the permutations $\tau$ and $\sigma$. The lengths of the orbits  of the commutator $[\tau,\sigma](i)$  give the angles at the singularities. Consequently $\tau$ and $\sigma$  commute if and only if there is no singularity for the flat metric which means that the square-tiled surface is a torus.
Moreover  the surface is connected if and only if the group generated by $\tau$ and $\sigma$ acts transitively on $\{1, 2, \dots, d\}$.
 A very good introduction to square-tiled surfaces can be found in Zmiaikou \cite{zm}.

\begin{figure}[h!]\label{fig:square-tiled-surface}

\begin{center}
\begin{tikzpicture}[scale = 1.8]

\draw (0,0) rectangle(1,  1);\draw (1,0) rectangle(2,  1);\draw (0,1) rectangle(1,  2);

\draw (0.5,0.5) node{1};
\draw (1.5,0.5) node{2};
\draw (0.5,1.5) node{3};
\end{tikzpicture}
\caption{ Square-tiled surface with 3 squares \\
$\tau (1,2,3) = (2,1,3)$ and  $\sigma (1,2,3) = (3,2,1)$}
\end{center}

\end{figure}
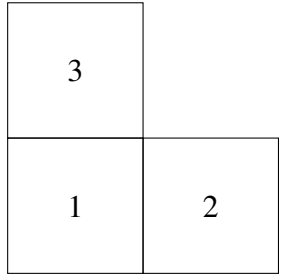

\section{Interval exchange transformation associated to square-tiled surfaces} \label{sec:square-tiled-interval exchange}
\subsection{Generalities}\label{sgen}

As we already noticed in the introduction, a square-tiled interval exchange transformation is the first return map on the diagonal of slope $-1$ of the linear flow on a square-tiled surface. Let  $p_l = \sigma^{-1}$ and $p_r = \tau^{-1}$, we first give a combinatorial definition of the square-tiled interval exchange transformation $T = T_{\alpha}$.

\begin{definition}\label{dsqiet} A {\em square-tiled $2d$-interval exchange transformation } with angle $\alpha$ and permutations  $p_l$ and $p_r$ is the exchange on $2d$ intervals defined by the positive vector $(1-\alpha,\alpha,1-\alpha,\alpha, \ldots,1-\alpha,\alpha)$ and permutation defined by $\pi (2i-1)=2p_l^{-1}(i)$, $\pi (2i)=2p_r^{-1}(i)-1$, $1\leq i\leq d$.\end{definition}

\begin{figure}[h] \label{fig:square-tiled-interval exchange}
\begin{center}
\begin{tikzpicture}[scale = 4]

\draw (0,0) rectangle(1,  1);\draw (1,0) rectangle(2,  1);\draw (0,1) rectangle(1,  2);


\draw(1,0)--(0,1);

\draw(2,0)--(0,2);

\draw[dashed](1,2)--(0.3,1.7);

\draw[dashed](2,1)--(1.3,0.7);

\draw[dashed](1,1)--(0.3,0.7);

\draw (0.3,0.7) node[above]{$1_l$};
\draw (0.8,0.2) node[above]{$1_r$};

\draw (1.3,0.7) node[above]{$2_l$};
\draw (1.8,0.2) node[above]{$2_r$};

\draw (0.3,1.7) node[above]{$3_l$};
\draw (0.8,1.2) node[above]{$3_r$};

(

\draw[dashed, <->](0.95,-0.05)--(0.25,0.65);
\draw (0.4,0.4) node[below]{$\sqrt{2}\alpha$};

\end{tikzpicture}\\
\caption{ Square-tiled interval exchange associated to the surface with 3 squares \\
$p_r^{-1} (1,2,3) = (2,1,3)$ and  $p_l^{-1} (1,2,3) = (3,1,2)$}

\end{center}

\end{figure}
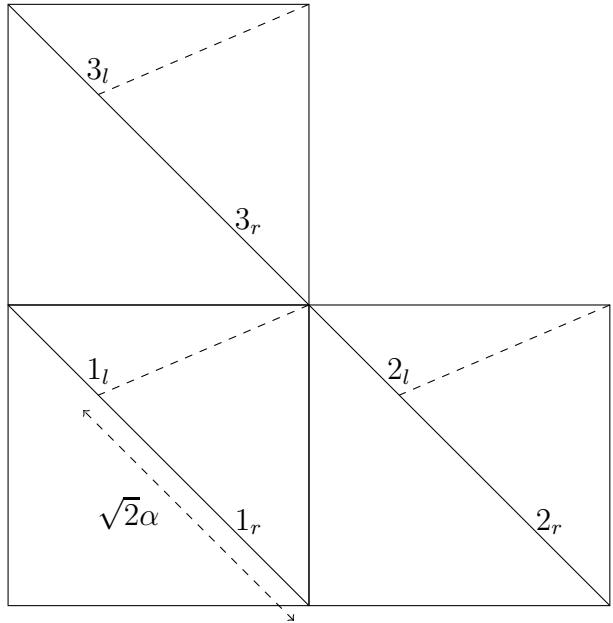

 Note that everything in this paper remains true if we replace the $\Delta_{2i-1}=[i-1, i-\alpha[$ by $[a_i, a_i+1-\alpha[$ and the $\Delta_{2i}=[i-\alpha,i[$ by $[a_i+1-\alpha, a_i+1[$ for some sequence satisfying $a_i\leq a_{i+1}-1$, and reorder the intervals in the same way. To avoid unnecessary complication, we shall always use $a_i=i-1$ as in the definition.\\
 
 Thus the discontinuities of $T$ are  some (not necessarily all, depending on the permutation) of the $\gamma_{2i-1}=i-\alpha$,  $1\leq i \leq d$, $\gamma_{2i}=i$, $1\leq i \leq d-1$, the discontinuities of $T^{-1}$ are  some of $\beta_{2i-1}=i-1+\alpha$,  $1\leq i \leq d$, $\beta_{2i}=i$, $1\leq i \leq d-1$.\\

We recall a classical result on minimality.

\begin{proposition} \label{min} Let $T$ be a square-tiled interval exchange transformation with irrational $\alpha$; $T$ is aperiodic, and is minimal if and only if there is no strict subset of $\{1\ldots d\}$ invariant by $p_l$ and $p_r$. 
\end{proposition}

\begin{proof}
Let $X$ be the square-tiled surface corresponding to $T$. As we already remarked in Section \ref{sec:square-tiled-surfaces}, the hypothesis on the permutations means that the surface $X$ is connected. A square tiled surface satisfies the Veech dichotomy (see \cite{ve89}). Thus the flow in direction $\theta$ is either periodic or minimal and uniquely ergodic. For square-tiled surfaces, periodic directions have rational slope. thus we get the result for the interval exchange transformation. 
\end{proof}

\begin{remark}
For square-tiled surfaces the whole strength of the Veech dichotomy is not needed and the result is already contained in \cite{ve87}.
Also notice that for square-tiled interval exchange transformations minimality implies unique ergodicity by \cite{bo}; we denote by $\mu$ the unique invariant  measure for $T$, namely the Lebesgue measure, and it is ergodic for
$T$.
\end{remark}

Let $T$ be a square-tiled interval exchange transformation. If we denote by $(x,i)$ the point $i-1+x$, then the transformation $T$ is defined on $[0,1[\times \{1,\ldots d\}$ by $T(x,i)=Rx, \phi_x(i)$ where $Rx=x+\alpha$ modulo $1$, and $\phi_x=p_l^{-1}$ if $x\in [0,1-\alpha[$, $\phi_x=p_r^{-1}$ if $x\in [1-\alpha,1[$. Thus $T$ is also a {\em $d$-point extension of a rotation}. This implies that $T$ has a rotation as a topological factor and thus a continuous eigenfunction, either for the topology of the interval or for the natural coding.

Note that in general our square-tiled interval exchanges, even when they are minimal, do not satisfy the usual {\em i.d.o.c. condition}; but when $\alpha$ is irrational, in cases when not all the $\gamma_i$ and $\beta_j$ are discontinuities, a square-tiled interval exchange on $2d$ intervals may indeed  be an i.d.o.c. interval exchange on a smaller number of intervals; to our knowledge this was first remarked by Hmili \cite{hh}, who uses some square-tiled interval exchanges (though they are not described as such) to provide examples of i.d.o.c. interval exchanges with continuous eigenfunctions; indeed, her simplest example is the one in Figure 2 above, which is  an i.d.o.c. $4$-interval exchange with permutation $\pi(1,2,3,4)=(4,2,1,3)$; as $3$-interval exchanges are topologically weak mixing, this ranks among the counter-examples to that property with the smallest number of intervals.

\subsection{Coding of a square-tiled interval exchange transformation}
We look now at the natural coding of $T$, which we denote again by $(X,T)$,  but with a change of notation:
 we denote by $i_l$ the letter $2i-1$ and by $i_r$ the letter $2i$, $1\leq i\leq d$. For any 
 (finite or infinite) word  $u$ on the alphabet $\{i_l,i_r, 1\leq i\leq d\}$, we denote by $\phi (u)$ the 
 sequence deduced from $u$ by replacing each $i_l$ by $l$, each $i_r$ by $r$. For a trajectory $x$ 
 for $T$ under our version of the  natural coding, $\phi(x)$ is a trajectory for $R$ under the coding
 by the partition into two atoms $[l]=[0,1-\alpha[\times \{1,\ldots d\}$, $[r]=[1-\alpha,1[\times \{1,\ldots d\}$, thus it is a trajectory for $R$ under its natural coding (as an exchange transformation of two intervals), and that is called  a Sturmian sequence.

\begin{lemma}\label{hom} For any word $w$ in $L(T)$, there are exactly $d$ words $v$ such that $\phi (w)=\phi (v)$, and for each of this words either $v=w$ or $\bar d(w,v)=1$.\end{lemma}
{\bf Proof}\\
Using the definition of $T$, we  identify the words of length $2$ in $L(T)$:
\begin{itemize}
\item if $\alpha<\frac{1}{2}$, for $1\leq i\leq d$, $i_l$ can be followed by $(p_ri)_l$ and $(p_ri)_r$,  $i_r$ can be followed by $(p_li)_l$;
\item if $\alpha>\frac{1}{2}$,  for $1\leq i\leq d$, $i_r$ can be followed by $(p_li)_r$ and $(p_li)_l$,  $i_l$ can be followed by $(p_ri)_r$.
\end{itemize}
If $w=w_1\ldots w_t$, then $w_i=(u_i)_{s_i}$ with $u_i \in \{1,\ldots ,d\}$ and 
$s_i\in \{l,r\}$, and $u_{i+1}=\pi_i(u_i)$ with $\pi_i\in \{p_l,p_r\}$. The above list of words of length $2$ implies that $\pi_i$ depends only on $s_i$; thus two homologous (= having the same image by $\phi$) words which  have the same $s_i$, have also the same $\pi_i$. Thus the words $v=v_1\ldots v_t$  homologous to $w$ are such that $v_1=x_{s_1}$, $v_i=(\pi_{i-1}\ldots \pi_1 (x))_{s_i}$ for $i>1$, thus there are as many such words as possible letters $x$, and if $x\neq u_1$ then $v_i\neq w_i$ for all $i$ as $\pi_{i-1}\ldots \pi_1$  are bijections.\qed\\

{\em Henceforth we shall make all computations for $\alpha<\frac{1}{2}$}; the complementary case gives exactly the same results, mutatis mutandis.\\

To understand the coding of $T$, we need a complete knowledge of the Sturmian coding of $R$; the one we quote here uses 
 a different version of the classic Euclid algorithm, which is  the self-dual induction of \cite{fz1} in the particular case of two intervals; all what we need to know is contained in the following proposition, which can also be proved directly without difficulty.

\begin{proposition}\label{sdr}
Let the Euclid continued fraction expansion of $\alpha<\frac{1}{2}$ be $[0,a_1+1, a_2,...]$, and the $q_k$, $k\geq 0$, be the denominators of the convergents of $\alpha$. We build inductively real numbers $l_n$ and $r_n$ and words $w_n$, $M_n$, $P_n$ in the following way: 
$l_1=\alpha$, $r_1=1-2\alpha$, $w_{1}=l$, $M_{1}=l$, $P_{1}=rl$. Then
      \begin{itemize}
\item whenever $l_n>r_n$, $l_{n+1}=l_{n}-r_{n}$,
$r_{n+1}=r_{n}$,
$w_{n+1}=w_{n}P_{n},$
$P_{n+1}=P_{n},$ 
$M_{n+1}=M_{n}P_{n};$
\item whenever $r_n>l_n$,  $l_{n+1}=l_{n}$,
$r_{n+1}=r_{n}-l_n$,
$w_{n+1}=w_{n}M_{n},$
$P_{n+1}=P_{n}M_{n},$
$M_{n+1}=M_{n}.$ 
\end{itemize}
Then $r_n>l_n$ for $1\leq n\leq a_1-1$,
$r_n<l_n$ for $a_1\leq n\leq a_1+a_2-1$, and so on.
 The $w_n$ are all the nonempty bispecial words of $L(R)$, $w_{n+1}$ being the shortest bispecial word beginning with $w_{n}$; moreover, $M_n$ and $P_n$ constitute all the {\em return words} of $w_n$ (namely, words $Z$ such that $w_n$ occurs exactly twice in $w_nZ$, once as a prefix and once as a suffix). \end{proposition}
 
 $\alpha$ has {\em bounded partial quotients} if the $a_i$ are bounded.\\

 The following lemma is also well known, but we did not find a proof in the existing literature.

\begin{lemma}\label{ebs} The words defined in Proposition 
 \ref{sdr} satisfy for all $n$
 \begin{itemize}
 \item $\ab P_{n} \ab + \ab M_{n}\ab=\ab w_n\ab +2$,
 \item $P_nM_n$ and $M_nP_n$ are right extensions of $P_1M_1$ and $M_1P_1$ by the same word.
 \end{itemize}

\noindent For $n\geq a_1+1$,  $w_n$ 
has exactly two extensions of length $\ab w_n\ab + \ab P_{n} \ab \wedge \ab M_{n}\ab$, and these
are   $w_nlrV_n$ and $w_nrlV_n$ for the same word $V_n$.

\noindent If $\alpha$ has bounded partial quotients,
there exists a constant $K_1$ such that $\ab P_{n} \ab \wedge \ab M_{n}\ab>K_1\ab w_n \ab$ for all $n$. 
\end{lemma}
{\bf Proof}\\
The first two assertions come from the recursion formulas.
 Then $n>a_1$  ensures that $\ab M_n \ab$ and $\ab P_n \ab$ are at
 least $2$; hence 
two possible extensions of $w_n$ of length $\ab w_n \ab+\ab P_{n} \ab \wedge \ab M_{n}\ab$ are the prefixes of that length of 
$w_nM_n$ and $w_nP_n$, hence of $w_nM_nP_n$ and $w_nP_nM_n$, thus they are of the form $w_nlrV_n$ and $w_nrlV_n$. Moreover, as there are no right special words 
in $L(R)$ sandwiched between $w_n$ and $w_nM_n$ or $w_nP_n$, there are only two extensions of that 
length of $w_n$, which proves the third assertion.\\

Let $K_0$ be the maximal value of the partial quotients of $\alpha$; because of the recursion formulas, 
at the beginning of a string of $n$ with $l_n<r_n$, we have $\ab P_{n} \ab < \ab M_{n} \ab$, then
for every $n$ in that string except the first one, and for the $n$ just after the end of that string, 
$\ab M_{n} \ab<\ab P_{n} \ab< (K_0+1)\ab M_{n} \ab$, and mutatis mutandis for strings  of $n$ with $l_n>r_n$. Thus we 
get the last assertion from the first one.
\qed\\

\section{Proof of Theorem \ref{thm:main}} \label{sec:proof}

\begin{proposition}\label{lmr} 
If  $\alpha$ has bounded partial quotients, there exists $C$ such that, for any 
integer $e$ with $e\leq d$ and $e\geq 1+\#\{i;p_lp_ri\neq p_rp_li\}$, if $v_i$ and $v'_i$, $1\leq i\leq e$, are words  in $L(T)$, of equal length $q$ such that
\begin{itemize}
 \item 
$\sum_{i=1}^e \bar d(v_i,v'_i)< C$,
\item $\phi (v_i)$ is the same word $u$ for all $i$,
\item $\phi (v'_i)$ is the same word $u'$ for all $i$,
\item $v_i\neq v_j$ for $i\neq j$.
\end{itemize}
Then $\{1,\ldots q\}$ is the disjoint union of three (possibly empty) integer intervals $I_1$, $J_1$, $I_2$ (in increasing order) such that 
\begin{itemize}
\item
$v_{i,J_1}=v'_{i,J_1}$ for all $i$, 
\item $\sum_{i=1}^e\bar d (v_{i,I_1},v'_{i,I_1})\geq 1$ if $I_1$ is nonempty, 
\item $\sum_{i=1}^e\bar d (v_{i,I_2},v'_{i,I_2})\geq 1$ if $I_2$ is nonempty,
\end{itemize}
where $z_{i,H}$ denotes the word made with the $h$-th letters of the word $z_i$ for all $h$ in $H$. \\
This implies in particular that $\# J_1\geq 1- \sum_{i=1}^e \bar d(v_i,v'_i)$. \end{proposition}
\begin{proof}  
We compare first $u$ and $u'$;  note that if we see $l$, resp. $r$, in some word $\phi (z)$ we see
some $i_l$, resp. $j_r$, at the same place on $z$; 
 thus $\bar d (z,z')\geq \bar d(\phi(z),\phi(z'))$ for all $z$, $z'$; in particular, if $
 \bar d(u,u')=1$, then $\bar d (v_i, v'_i)=1$ for all $i$ and our assertion is proved.
 
 Thus we can assume $\bar d(u,u')<1$.
 We partition $\{1,\ldots q\}$ into
successive integer intervals where $u$ and $u'$ agree or disagree: we get intervals $I_1$, $J_1$,
\ldots, $I_r$, $J_s$, $I_{s+1}$, where $r$ is at least $1$, the intervals are nonempty except 
possibly for $I_1$ or $I_{s+1}$, or both, and for all $j$, $u_{J_j}=u'_{J_j}$, and, except if $I_j$ is empty, $u_{I_j}$ and $u'_{I_j}$ are completely different, i.e.
their distance $\bar d$ is one.

Then for $i\leq s-1$, the word $u_{J_i}=u'_{J_i}$ is right special in the language $L(R)$ of the rotation, and this word is left special if $i\geq 2$.\\

{\em $(H0)$ We suppose first that $u_{J_1}=u'_{J_1}$ is also left special and $u_{J_r}=u'_{J_r}$ is also right special}.\\

Then all the $u_{J_i}=u'_{J_i}$ are bispecial; thus, for a given $i$,  $u_{J_i}=u'_{J_i}$  must be some $w_n$ of  Proposition \ref{sdr}; then Lemma \ref{ebs} implies that 
 either $\# J_j$ is smaller than a fixed $m_1$,  or 
 $\#I_{j+1}=2$ and $$\#I_{j+1}+
\#J_{j+1}>K_1\ab w_{n} \ab\geq K_1 \#J_j,$$

Similar considerations for $R^{-1}$ imply that for $j>1$ either $\# J_j< m_1$, or $\#I_{j}=2$ and 
$\#J_{j-1}+\# I_j>K_1 \#J_j$.

Note that this does not give any conclusion on $\bar d(u,u')$, and
indeed everybody knows $R$ is rigid, and thus admits a lot of $\bar d$-neighbours.\\

We look now at the words $v_i$ and $v'_i$ for some $i$; by the remark above, $v_{i,I_j}$ and $v'_{i,I_j}$ are completely different if $I_j$ is nonempty.
As for $v_{i,J_j}$ and $v'_{i,J_j}$, they have the same image by $\phi$, thus by Lemma \ref{hom} they are equal if they begin by the same letter, 
completely different otherwise.

Moreover, suppose that $J_j$ has length at least $m_1$, and $v_{i,J_j}=v'_{i,J_j}=z_i$, ending with
the letter $s_i$: because of Lemma \ref{ebs} applied to $\phi (z_i)$.
and taking imto account 
the possible words of length $2$ in $L(T)$, $z_i$ has two extensions of length $\ab z_i\ab +3$ in $L(T)$, and 
they are $z_i(p_rs_i)_l(p_rp_rs_i)_r(p_lp_rp_rs_i)_l$ 
and 
$z_i(p_rs_i)_r(p_lp_rs_i)_l(p_rp_lp_rs_i)_l$, which gives us the first letters of the two words $v_{i,J_{j+1}}$ and $v'_{i,J_{j+1}}$.\\

We estimate $c=\sum_{i=1}^e \bar d(v_i,v'_i)$, by looking at 
the indices in some set 
$G_j=
J_j\cup I_{j+1}\cup J_{j+1}$, for any $1\leq j\leq r-1$;
\begin{itemize}
\item if both $\#J_j$ and $\#J_{j+1}$ are
smaller than $m_1$ the contribution of $G_j$ to the sum $c$ is at least $\frac{1}{2m_1+1}$ as
$I_{j+1}$ is nonempty by construction;
\item if $\#J_j\geq m_1$, 
 and 
for at least one $i$ $v_{i,J_j}$ and $v'_{i,J_j}$ are completely different, then the contribution of $G_j$ to $c$
is bigger than $\frac{1}{2} \wedge \frac{K_1}{K_1+1}$ as either $\# J_{j+1}<m_1$ or 
$\# J_j+\# I_{j+1}>K_1\# J_{j+1}$;

\item if 
$\#J_j\geq m_1$ and
for all $i$, $v_{i,J_j}=v'_{i,J_j}=z_i$; then, because the $v_i$ are all
different and project by $\phi$ on the same word,  the last letter $s_i$ of $z_i$
takes $e$ different values when $i$ varies; thus  
$p_rp_lp_rs_i\neq p_lp_rp_rs_i$ for at least one $i$, and this ensures that for this $i$, $v_{i,J_{j+1}}$ and $v'_{i,J_{j+1}}$
 are completely different. As $\#J_{j+1}+\#I_{j+1}>K_1\#J_j$, the contribution of $G_j$ to $c$ 
is bigger than  $\frac{K_1}{K_1+1}$;

\item if $\#J_{j+1}\geq m_1$, we imitate the last two items by looking in the other direction.

\end{itemize}

Now, if $s$ is even, we can cover $\{1,\ldots q\}$ by sets $G_j$ and some internediate $I_l$, and get 
that $c$ is at least a constant $K_2$. If $s$ is odd and at least $3$,
by deleting either $I_1$ and $J_1$, or $J_s$ and $I_{s+1}$, we cover at least half of $\{1,\ldots q\}$ 
by sets $G_j$ and some internmdiate $I_l$, and $c$ is at least $\frac{K_2}{2}$.\\

Thus if $\sum_{i=1}^e \bar d(v_i,v'_i)$ is smaller than a constant $K_3$, we must have $s=1$; then if 
$\sum_{i=1}^e \bar d(v_i,v'_i)<1$, $v_{i,J_{1}}=v'_{i,J_{1}}$. Thus we get our conclusion if $c<C=K_3\wedge 1$, under the extra hypothesis $(H0)$.\\\\

If $(H0)$ is not satisfied,  we modify the $v_i$ and $v'_i$ to $\tilde v_i$ and $\tilde v'_i$ to get it.

Note that if $u_{J_1}=u'_{J_1}$ is not left special, then $I_1$ is empty, and $u$ and $u'$ are uniquely extendable to the left, and by the same letter; we continue to extend uniquely to the left as long as the extension of $u_{J_1}=u'_{J_1}$ remains not left special, and this will happen until we have extended $u$ and $u'$ (by the same letters)  to a length $q_0$. As for 
$v_{i,J_{1}}$ and $v'_{i,J_{1}}$, they  are either equal or completely different; then
\begin{itemize}
\item if for at least one $i$ $v_{i,J_{1}}$ and $v'_{i,J_{1}}$ are completely different, we delete the prefix $v_{i,J_{1}}$ from every $v_i$, the prefix $v'_{i,J_{1}}$ from every
  $v'_i$;
\item if for all $i$ $v_{i,J_{1}}=v'_{i,J_{1}}$; then  $v_i$ and $v'_i$ are uniquely extendable to the left, and by the same letter,  as long as $u$ and $u'$ are; then for all $i$, we take the unique left extensions of
length $q_0$ of  $v_i$ and $v'_i$.
\end{itemize}

If $u_{J_s}=u'_{J_s}$ is not right special, we do the same operation on the right; thus we get new pairs of words $\tilde v_i$ and $\tilde v'_i$, of length $\tilde q$. In building them, 
we have added no difference (in the sense of counting $\bar d$) between $v_i$ and $v'_i$, but have possibly deleted a set of $q_1$ indices which gave a contribution at least one to the sum $c$ and thus created at least $q_1$ of these differences, while when we extend the words we can only decrease the distances $\bar d$; thus if $c<C\leq 1$, 
$\sum_{i=1}^e \bar d(\tilde v_i,\tilde v'_i)\leq \frac{qc-q_1}{q-q_1}\leq c$. Then our pairs satisfy all the conditions of the part we have already proved (the $\tilde v_i$ are all different because they are different on at least one letter and have the same image by $\phi$).

Thus $\{1,\ldots \tilde q\}$ is partitioned into $\tilde I_1$, $\tilde J_1$, $\tilde I_2$, with the properties in the conclusion of the proposition. \\

We go back now to the original $v_i$ and $v'_i$. 
\begin{itemize}
\item Suppose first that to get the new words we have either shortened or not modified the $v_i$ on the left, and either shortened and not modified the $v_i$ on the right: then we get our conclusion with $J_1$ a translate of $\tilde J_1$, $I_1$ the union of a translate of $\tilde I_1$ and an interval $I_0$ corresponding to a part we have cut,
 $I_2$ the union of a translate of $\tilde I_2$ and an interval $I_3$ corresponding to a part we have cut. 
\item Suppose that to get the new words we have either shortened or not modified the $v_i$ on the left, and lengthened the $v_i$ on the right: then we get our conclusion with $J_1$ a translate of a  nonempty subset of $\tilde J_1$, $I_1$ the union of a translate of $\tilde I_1$ and an interval $I_0$ corresponding to a part we have cut,
 $I_2$ empty as $\tilde I_2$. 
\item A symmetric reasoning applies if to get the new words we have either shortened or not modified the $v_i$ on the rightt, and lengthened the $v_i$ on the left.
\item Suppose that to get the new words we have lengthened the $v_i$ on the right and on the left: then we get our conclusion with $J_1$ a translate of a nonempty subset of $\tilde J_1$, $I_1$ empty as $\tilde I_1$,
 $I_2$ empty as $\tilde I_2$. 
 \end{itemize}
\end{proof}

\begin{remark}\label{ctex}
Our proposition is not valid for $e\leq \#\{i;p_lp_ri\neq p_rp_li\}$:
if we take 
$v_i$ and $v'_i$ such that $\phi (v_i)=w_nlry_n$, $\phi (v'_i)=w_nrly_n$, and 
that the $\ab w_n \ab$-th letter of  $v_i$ and $v'_i$ is $s_i$ where $p_rp_lp_rs_i=p_lp_rp_rs_i$, then the  $v_i$ and $v'_i$ 
do not satisfy the conclusion if $y_n$ and $w_n$ are of comparable lengths,
though $\sum \bar d (v_i, v'_i)\leq d\frac{2}{\ab w_n \ab +\ab y_n\ab +2}$ may
be arbitrarily small.
\end{remark}

We now prove the hard part of Theorem \ref{thm:main} from Proposition \ref{lmr}.

\begin{proof}
We look at the $2d$ intervals $\Delta_i$ giving the natural coding.

Assume that $(X,T)$ is rigid; then there exists a sequence $q_k$ tending to infinity such that 
$\mu (\Delta_i\Delta T^{q_k}\Delta_i)$ tends to zero for $1\leq i\leq 2d$. 

We fix $\epsilon<\frac{C}{2d^2}$, and
$k$ such that for all $i$
$$\mu (\Delta_i\Delta T^{q_k}\Delta_i)<\epsilon.$$

Let $A_i=\Delta_i\Delta T^{q_k}\Delta_i$; 
by the ergodic theorem, $\frac{1}{m}\sum_{j=0}^{m-1} 1_{T^jA_i}(x)$ tends to $\mu (A_i)$, for almost all $x$ (indeed
for all $x$ because $(X,T)$ is uniquely ergodic). Thus for all $x$, there exists $m_0$ such that for all $m$ larger 
than some $m_0$ and 
all $i$, 
$$\frac{1}{m}\sum_{j=0}^{m-1} 1_{T^jA_i}(x)<\epsilon.$$ By summing these $2d$ inequalities, we get that 
$$\bar d(x_0\ldots x_{m-1}, x_{q_k}\ldots x_{q_k+m-1})<2d\epsilon$$ for all $m>m_0$. Moreover, given an $x$, we can
choose $m_0$ 
such that for all $m>m_0$ these inequalities are satisfied if we replace $x$ by any of the $d$ different points $x^i$ 
such that $\phi (x^i)=\phi (x)$.\\

We choose such an $x$, and apply Proposition \ref{lmr} to $e=d$ and the words $v_i=(x^i)_0, \ldots ,(x^i)_{m-1}$,
$v'_i=(x^i)_{q_k}, \ldots ,(x^i)_{q_k+m-1}$. As
 we know that $c$  is smaller than $2d^2\epsilon$, we get that for  any
$m>m_0$, the words $(x_0\ldots x_{m-1})$ and $(x_{q_k}\ldots x_{q_k+m-1})$ must coincide on a connected
part
larger than
$m$ multiplied by a constant; thus
 $x_l\ldots x_{p-1}$ and $x_{q_k+l}\ldots x_{q_k+p-1}$ coincide for some fixed $l$ and all $p$ 
 large enough,
but this implies that there is a periodic point, which has been disproved in  Proposition \ref{min}.
 
\end{proof}\

The other direction of Theorem \ref{thm:main}  is already known, but we include it with a short proof using our combinatorial methods.

\begin{proposition}\label{rig2} Let $T$ be a minimal square-tiled interval exchange transformation such that $\alpha$ is irrational and has unbounded partial quotients; then $(X,T,\mu)$ is  rigid.
\end{proposition}
\begin{proof}
For all $n$, the trajectories of the rotation are covered by disjoint occurrences of $M_n$ and $P_n$ (of Proposition \ref{sdr}) as these are the return words of $w_n$.
Suppose for example $l_m>r_m$ for $b_n\leq m \leq b_n+a_n-1$; then because of the previous step $\ab P_{b_n} \ab>\ab M_{b_n}\ab$; then $P_{b_n+a_n}=P_{b_n}$, $M_{b_n+a_n}=M_{b_n}P_{b_n}^{a_n}$,
$P_{b_n+a_n+1}=P_{b_n}M_{b_n}P_{b_n}^{a_n}$, $M_{b_n+a_n+1}=M_{b_n}P_{b_n}^{a_n}$. Hence disjoint occurences of the word $P_{b_n}^{a_n}$ fill a proportion at least $\frac{a_n}{a_n+2}$ of the length of both
$M_{b_n+a_n}$ and $P_{b_n+a_n}$. The trajectories for $T$ are covered by the $d$ words $P_{n,i}$ and $M_{n,i}$ which project on $P_n$ and $M_n$ by $\phi$, and a proportion at least $\frac{a_n}{a_n+2}$ of them are covered by  disjoint occurrences of the $d$ words which project by $\phi$ on $P_{b_n}^{a_n}$. Each $P_{n,i}$ can be followed by exactly one $P_{n,j}$, and thus the $P_{n,i}$, $1\leq i\leq d$, are grouped into $d'_n\leq d$ cycles
$P_{n,i_{n,j,1}}\ldots P_{n,i_{n,j,c_{n,j}}}$, $1\leq j\leq d'_n$, $1 \leq c_{n,j}\leq d$, where for a given $n$ all the possible $P_{n,i_{n,j,l}}$ are different and the only $P_{n,h}$ which can follow
$P_{n,i_{n,j,c_{n,j}}}$ is $P_{n,i_{n,j,1}}$. Let $s_n\leq d^d$ the least common multiple of all the $c_{b_n,j}$, $1\leq j\leq d'_{b_n}$,; then if we move by $T^{s_n\ab P_{b_n}\ab}$ inside one of the words which project on
$P_{b_n}^{a_n}$, we see the same letter. Thus, if $E$ is a fixed cylinder of length $L$, $\mu (E\Delta T^{s_n\ab P_{b_n}\ab}E)$ is at most $\frac{2}{a_n}+\frac{s_n}{a_n}+\frac{L}{\ab P_{b_n}\ab}$. Thus, possibly replacing $P$ by $M$ for the cases $l_m<r_m$, we get thet  if the $a_n$ are unbounded $T$ is rigid, as the cylinders for the natural coding generate the whole $\sigma$-algebra. 
\end{proof}

\section{Proof of Theorem \ref{thm:flow} and Theorem \ref{thm:rank}} \label{sec:consequences}
\subsection{Rigidity of the flow}
\begin{proposition}
Let $X$ be a square-tiled surface and $\theta$ a direction, $S_t$ the linear flow in direction $\theta$ and
$T = T_\alpha$ the associated interval exchange transformation. 
The flow $S_t$ is rigid whenever $T$ is rigid.
\end{proposition}
\begin{proof}

 The key point is that the flow $S_t$ is a suspension flow over $T$ with {\it constant} roof function. Denote by $I$ the union of the diagonals of slope $-1$. In fact, if a point belongs to $I$, the return time $\rho$ to $I$ is  independent of the point since diagonals are parallel (see for instance Figure  2).

 Now, suppose $T$ is rigid; if $q_n$ is a rigidity sequence for $T$, then $\rho q_n$ is a rigidity sequence
 for the flow $S_t$, 
 and thus $S_t$ is rigid.
 
 Suppose the flow is rigid, with rigidity sequence $Q_n$; let $Q_n = \rho Q'_n$. Denote by $q_n$ the nearest integer to $Q_n'$. Since the return time $\rho$ is constant,  $Q_n'$ is close to the integer $q_n$: looking at the projection in the torus $\GR^2/\GZ^2$, a point in $I$ cannot be close to $I$ otherwise. Thus, as $Q_n$ is a rigidity time for the flow, $q_n$ is a rigidity time for $T$.
 
 \end{proof}

\subsection{Rank}
We now prove Theorem \ref{thm:rank}.
\begin{proof}
If $T$ is of rank one, its natural coding satisfies the non-constructive symbolic definition of rank one, see the survey \cite{fr1}: for every  positive
$\epsilon$, for every natural integer $l$, there exists a word $B$ of
length $\ab B\ab$ bigger or equal to $l$ such
that, for all $n$ large enough, on a subset of $X$ of measure
 at least
$1-\epsilon $, the prefixes of length $n$ of the trajectories are of the form
$\delta_1B_{1}...\delta_pB_{p}\delta_{p+1}$, with
$\ab \delta_1\ab +...\ab \delta_p \ab<\epsilon n$ and $\overline{d}(W_i,B)
< \epsilon$ for all $i$. But then 
Proposition \ref{lmr} is valid for $e=1$ and implies, possibly after shortening $B$ by a prefix and a suffix of total relative length at most $\epsilon$, and  lengthening the $\delta_i$ accordingly, that the same is true with $B_i=B$ fo all $i$. By projecting by $\phi$, we get a similar structure for the trajectories of the rotation $R$. Such a structure for $R$ implies  that the quantity $F$ defined in Definition 4 of \cite{che}   is equal to $1$, and by Proposition 5 of that paper this  is impossible when $\alpha$ has bounded partial quotients.
\end{proof}

\section{Interval exchange transformations associated to billiards in Veech triangles} \label{sec:reg-billiards}
We consider the famous examples of \cite{ve89}: unfolding the billiard in the right-angled triangle with angles $(\pi/n, \pi/2, (d-1)\pi/2d)$, one gets a regular double $2d$-gon. A path, which starts in the interior of the polygon, moves with constant velocity until it hits the boundary, then it re-enters the polygon at the corresponding point of the parallel side, and continues travelling with the same velocity.

 We follow the presentation of \cite{su}.
The sides of the $2d$-gon are labelled $A_1$, ..., $A_d$ from top to bottom on the right, and two parallel sides have the same label. We draw 
the diagonal from the right end of the side labelled $A_i$ on the right to the left end of the side labelled $A_i$ on the left. There always exists $i$ such that the angle $\theta$ between the billiard direction and the orthogonal of this diagonal is between $\frac{-\pi}{2d}$ and $\frac{\pi}{2d}$ (see Figure 3) .

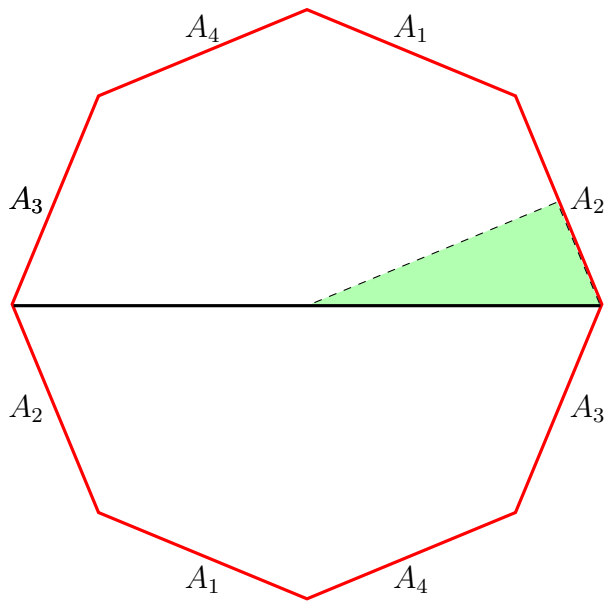
\begin{figure}[h!] \label{fig:octagon}
\begin{center}
\begin{tikzpicture}[scale=1.5]
\draw(0,0)--++(22.5:1) node[below]{$A_4$};
\draw (0,0)--++(22.5:2)--++(67.5:1) node[right]{$A_3$};
\draw(0,0)--++(22.5:2)--++(67.5:2)--++(112.5:1) node[right]{$A_2$};
\draw (0,0)--++(22.5:2)--++(67.5:2)--++(112.5:2)--++(157.5:1) node[above]{$A_1$};
\draw (0,0)--++(22.5:2)--++(67.5:2)--++(112.5:2)--++(157.5:2)--++(202.5:1)node[above]{$A_4$};
\draw (0,0)--++(22.5:2)--++(67.5:2)--++(112.5:2)--++(157.5:2)--++(202.5:2)--++(247.5:1)node[left]{$A_3$};
\draw (0,0)--++(22.5:2)--++(67.5:2)--++(112.5:2)--++(157.5:2)--++(202.5:2)--++(247.5:2)--++(292.5:1)node[left]{$A_2$};
\draw (0,0)--++(22.5:2)--++(67.5:2)--++(112.5:2)--++(157.5:2)--++(202.5:2)--++(247.5:1)node[left]{$A_3$};
\draw (0,0)--++(22.5:2)--++(67.5:2)--++(112.5:2)--++(157.5:2)--++(202.5:2)--++(247.5:2)--++(292.5:2)--++(337.5:1)node[below]{$A_1$};
\draw[very thick, red] (0,0)--++(22.5:2)--++(67.5:2)--++(112.5:2)--++(157.5:2)--++(202.5:2)--++(247.5:2)--++(292.5:2)--++(337.5:2)--cycle;
\draw[dashed, fill=green!30!white](2.6,2.6)--(0,2.6)--++(22.5:2.4)--cycle;
\draw[very thick](-2.6,2.6)--(2.6,2.6);
\end{tikzpicture}
\caption{ Regular Octagon}
\end{center}
\end{figure}

 We put on the circle the points $-ie^{\frac{ij\pi}{d}}$ from $j=0$ to $j=d$, which are the vertices of the $2d$-gon; our diagonal is the vertical line from $-i$ to $i$, we project on it the sides of the polygon which are to the right of the diagonal, partitioning it into intervals $I_1$, ... $I_d$, and the sides of the polygon which are to the left of the diagonal, partitioning it into intervals $J_1$, ... $J_d$. The transformation which exchanges the intervals $(I_1,... I_d)$ with the $(J_1,... J_d)$ is identified with the interval exchange transformation ${\mathcal I}$ on $[-1,1[$ whose discontinuities are $\gamma_j=-\cos\frac{j\pi}{d}+\tan\theta\sin\frac{j\pi}{d}$, $1\leq j\leq d-1$, while the discontinuities of ${\mathcal I}^{-1}$ are $\beta_j=-\gamma_{d-j}$, composed with the map $x\to-x$ if $\theta<0$. ${\mathcal I}$ is a  $d$-interval exchange transformation with permutation $p$ defined by $p(j)=d-j+1$ (see Figure 4).

 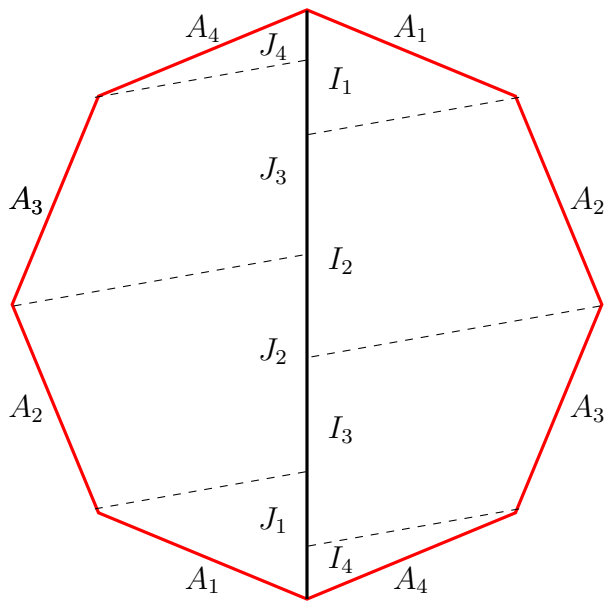
\begin{figure}[h!] \label{fig:interval exchange-octagon}
\begin{center}
\begin{tikzpicture}[scale=1.5]
\draw(0,0)--++(22.5:1) node[below]{$A_4$};
\draw (0,0)--++(22.5:2)--++(67.5:1) node[right]{$A_3$};
\draw(0,0)--++(22.5:2)--++(67.5:2)--++(112.5:1) node[right]{$A_2$};
\draw (0,0)--++(22.5:2)--++(67.5:2)--++(112.5:2)--++(157.5:1) node[above]{$A_1$};
\draw (0,0)--++(22.5:2)--++(67.5:2)--++(112.5:2)--++(157.5:2)--++(202.5:1)node[above]{$A_4$};
\draw (0,0)--++(22.5:2)--++(67.5:2)--++(112.5:2)--++(157.5:2)--++(202.5:2)--++(247.5:1)node[left]{$A_3$};
\draw (0,0)--++(22.5:2)--++(67.5:2)--++(112.5:2)--++(157.5:2)--++(202.5:2)--++(247.5:2)--++(292.5:1)node[left]{$A_2$};
\draw (0,0)--++(22.5:2)--++(67.5:2)--++(112.5:2)--++(157.5:2)--++(202.5:2)--++(247.5:1)node[left]{$A_3$};
\draw (0,0)--++(22.5:2)--++(67.5:2)--++(112.5:2)--++(157.5:2)--++(202.5:2)--++(247.5:2)--++(292.5:2)--++(337.5:1)node[below]{$A_1$};
\draw[very thick, red] (0,0)--++(22.5:2)--++(67.5:2)--++(112.5:2)--++(157.5:2)--++(202.5:2)--++(247.5:2)--++(292.5:2)--++(337.5:2)--cycle;

\draw[very thick](0,0)--(0,5.22);

\draw[dashed](1.88,0.8)--++(10:-1.95);
\draw[dashed](2.6,2.6)--++(10:-2.6);
\draw[dashed](1.88,4.45)--++(10:-1.95);
\draw[dashed](-1.88,4.45)--++(10:1.95);
\draw[dashed](-2.6,2.6)--++(10:2.6);
\draw[dashed](-1.88,0.8)--++(10:1.95);

\draw (0.3,0.35) node{$I_4$};
\draw (0.3,1.5) node{$I_3$};
\draw (0.3,3) node{$I_2$};
\draw (0.3,4.6) node{$I_1$};

\draw (-0.3,4.9) node{$J_4$};
\draw (-0.3,3.8) node{$J_3$};
\draw (-0.3,2.2) node{$J_2$};
\draw (-0.3,0.7) node{$J_1$};

\end{tikzpicture}
\caption{Interval exchange transformation in the regular octagon}

\end{center}
\end{figure}

Thus we consider the one-parameter family of interval exchange transformations $\mathcal I$, which depend on the parameter $\theta$, $\frac{-\pi}{2d}<\theta <\frac{\pi}{2d}$ or equivalently on the parameter 
$$y= \frac{1}{2}\left(\frac{\sin\frac{\pi}{d}}{|\tan\theta|}-\left(1+\cos\frac{\pi}{d}\right)\right)>0.$$

\subsection{A rigid subfamily of interval exchange transformations}

Let $\lambda=2\cos^2\frac{\pi}{2d}=1+\cos\frac{\pi}{d}$. We define an application $g$ by $g(y)=y-\lambda$ if $y>\lambda$,
$g(y)=\frac{y}{1-2y}$ if $0<y<\frac{1}{2}$ (the value of $g$ on other sets is irrelevant).

From Theorem 11 of \cite{ngon}, in the particular case of Theorem 13 of the same paper, we deduce the following result.
\begin{proposition}\label{rigng}
If $y$ is such that there exist two sequences $m_n$ and $q_n$, with $m_0=q_0=0$, and the iterates $g^{(n)}(y)$ satisfy 
\begin{itemize}
\item $\lambda< g^{(n)}(y)$ if $m_0+q_0+m_1+q_1+\ldots +m_k+q_k\leq n\leq m_0+q_0+m_1+q_1+\ldots +m_k+q_k+m_{k+1}-1$ for some $k$,
\item $0< g^{(n)}(y)\leq \frac{1}{2}$ if $m_0+q_0+m_1+q_1+\ldots +m_k+q_k+m_{k+1}\leq n \leq m_0+q_0+m_1+q_1+\ldots +m_k+q_k+m_{k+1}+q_{k+1}-1$ for some $k$,
\end{itemize}
then for all $n$, the trajectories of $\mathcal I$ are covered by disjoint occurrences of words $M_{n,i}$ and $P_{n,i}$, $1\leq i \leq d-1$,
built inductively  in the following way:
\begin{itemize}
\item $M_{0,i}=i$, $1\leq i\leq d-1$, $P_{0,1}=d1$, $P_{0,i}=i$, $2\leq i\leq d-1$;
\item if $m_0+q_0+m_1+q_1+\ldots +m_k+q_k\leq n\leq m_0+q_0+m_1+q_1+\ldots +m_k+q_k+m_{k+1}-1$ for some $k$,
$$P_{n+1,i}=P_{n,i} \quad\mbox{for}\quad  1\leq i\leq d-1,$$
$$M_{n+1,i}=M_{n,i}P_{n,d-i+1}P_{n,i}\quad\mbox{for}\quad  2\leq i\leq d,$$
$$M_{n+1, 1}=M_{n,1}P_{n,1};$$
\item if $m_0+q_0+m_1+q_1+\ldots +m_k+q_k+m_{k+1}\leq n \leq m_0+q_0+m_1+q_1+\ldots +m_k+q_k+m_{k+1}+q_{k+1}-1$ for some $k$,
$$M_{n+1,i}=M_{n,i} \quad\mbox{for}\quad  1\leq i \leq d-1,$$
$$P_{n+1,i}=P_{n,i}M_{n+1,d-i}M_{n+1,i} \quad\mbox{for}\quad 1 \leq i \leq d-1.$$

\end{itemize}
\end{proposition}

We can now state

\begin{proposition} \label{prop:interval exchange-billiard} There exists two functions $F$ and $G$ such that, if for infinitely many $n$ either $m_n>F(m_0, q_0, m_1,q_1, \ldots, m_{n-1}, q_{n-1})$ or $q_n>G(m_0, q_0, m_1,q_1, \ldots, m_{n-1}, q_{n-1},m_n)$, and $y$ is as in Proposition \ref{rigng}, then $\mathcal I$ is rigid.\end{proposition}
{\bf Proof}\\ If $m_n$ is large, as in Proposition \ref{rig2} we cover most of the trajectories by words $(P_{n,d-i+1}P_{n,i})^{m_n}$,$2\leq i\leq d-1$, and $P_{n,1}^{m_n}$. Let $s_n$ be the least common multiple  of $\ab P_{n,d-i+1} \ab + \ab P_{n,i}\ab$, $2\leq i\leq d-1$, and $\ab P_{n,1}\ab $; when we move by $s_n$ inside these words, we see the same letter; thus $s_n$ will give a rigidity sequence for $\mathcal I$ if all the 
$m_n(\ab P_{n,d-i+1} \ab + \ab P_{n,i}\ab)$, $2\leq i\leq d-1$, and $m_n\ab P_{n,1}\ab $ are much larger than $s_n$, which gives a condition as in the hypothesis; and similarly with the $M$ words if $q_n$ is large. \qed\\

\subsection{Rigidity of the flow}

\begin{proposition} \label{prop-flow-billiard} There exists a dense $G_\delta$ set of directions $\theta$, of positive Hausdorff dimension for which the flow is rigid.\end{proposition} 

\begin{proof}

We recall that in every non minimal direction, the linear flow is periodic (see \cite{ve89}). In a periodic direction, the surface is decomposed into parallel cylinders of commensurable moduli. Up to normaliization, the vectors of the heights of the cylinders form a finite set. More precisely, the periodic directions correspond to cusps of a lattice in $\textrm{SL}(2,\GR)$ (see \cite{ve89}).

We give a detailed proof in the case $d=4$  since one can make explicit computations. We recall that in a periodic direction, the octagon is decomposed into cylinders. The ratio of the lengths of these cylinders is $\sqrt{2}$. 

Let us fix a direction $\theta$.
 We 	approximate $\theta$ by periodic directions $\theta_n$. We denote by  $l_n$ the length of the shortest cylinder in direction $\theta_n$.
 We say that $\theta$ is approximable by $(\theta_n)$ at speed $a$ if 
  $\ab \theta -\theta_n \ab<\frac{1}{l_n^{2+a}}$. Assume that this property holds. Denote by $C_{1,n}$ the cylinder of length $l_n$ and $C_{2,n}$ the cylinder of length $l_n\sqrt{2}$. We approximate $\sqrt{2}$ by $\frac{p_n}{q_n}$, with $\ab \sqrt{2}-\frac{p_n}{q_n}\ab <\frac{1}{q_n^2}$.

Our rigidity sequence will be $p_nl_n$. As in Figure 5,  flowing in direction $\theta$, the subinterval $B$ of the interval $J$ of the cylinder $C_{1,n}$ that escapes the cylinder $C_{1,n}$ after time $l_n$ has length  $l_n\ab \theta-\theta_n\ab $. Thus the area of the sub rectangle that does not run along the  cylinder has measure $l_n^2\ab \theta-\theta_n\ab $.
After time $p_nl_n$, the part that escapes has measure
$p_nl_n^2\ab \theta-\theta_n\ab <\frac{p_n}{l_n^{a}}$. This measure tends to zero as $n$ tends to infinity if $p_n<<l_n^a$.

When we move by the time $p_nl_n$ of the flow inside $C_{1,n}$, there is no vertical translation by construction; inside $C_{2,n}$, we move by  $p_nl_n$ modulo $l_n\sqrt{2}$; but $p_nl_n=l_n(q_n\sqrt{2}+\frac{x_n}{q_n})$ with $\vert x_n \vert<1$, so we move by less than $\frac{l_n}{q_n}$. Thus rigidity holds if $l_n<<q_n$ or equivalently $l_n<<p_n$.

Our two conditions $l_n<<p_n<<l_n^a$ are compatible if $a>1$.
Moreover, since the periodic directions correspond to the cusps of a lattice in $\textrm{SL}(2,\GR)$, the set of $\theta$ approximable at speed $a$ has positive Hausdorff dimension and is a dense $G_\delta$ set of the unit circle. Nevertheless it has 0 measure.

For general $d$,  we have $d-2$ cylinders $C_{n,i}$ of lengths $l_n\tau_i$ with $\tau_1=1$. By Dirichlet, we find  $p_n$ and $q_{n,i}$, such that
$\ab \frac{1}{\tau_i}-\frac{q_{n,i}}{p_n}\ab <\frac{1}{p_n^{1+b}}$ for all $i>1$ where $b = \frac{1}{d-3}$. Thus $p_nl_n=l_n(q_{n,i}\tau_i+\frac{x_{n,i}\tau_i}{p_n^b})$ with $\vert x_{n,i} \vert <1$, thus $p_nl_n$ is a rigidity sequence if both $p_n<<l_n^a$
and $l_n<<p_n^b$ which is possible if $ab>1$ which means that $a > d-3$. 
\end{proof}

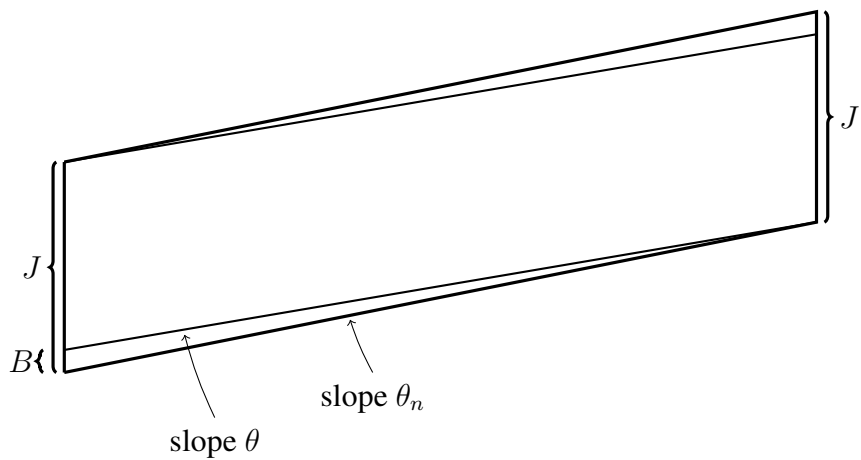
\begin{figure}[ht]
\centering
\begin{tikzpicture} [scale=2]
  \draw [very thick] (0,-0) -- (0,1.4);
  \draw [very thick] (0,0) -- (5,1) -- (5,2.4) -- (0,1.4);
  \draw [thick] (0,1.4) -- (5,2.25);
  \draw [thick] (0,0.15) -- (5,1);
  \draw [very thick,snake=brace,raise snake=3pt]
    (0,0) -- node [left=4pt] {$J$} (0,1.4);
  \draw [very thick,snake=brace,raise snake=8pt]
    (0,0) -- node [left=7pt] {$B$} (0,0.15);
  \draw [thin,->] (2.05,0) node [below] {slope $\theta_n$} to
    [bend left=5] (1.9,0.35);
  \draw [thin,->] (1,-0.3) node [below] {slope $\theta$} to
    [bend left=5] (0.8,0.25);
  \draw [very thick,snake=brace,raise snake=3pt]
    (5,2.4) -- node [right=4pt] {$J$} (5,1);
\end{tikzpicture}
\caption{ 
Trajectories in direction $\theta$ run along the cylinder from $J$ in
direction $\theta_n$ once unless they are in the subinterval $B$.}
\label{fig:cylinder}
\end{figure}

\end{document}